\documentclass[12pt]{article}

\usepackage[paperwidth=8.5in, paperheight=11in, top=1in, bottom=1in, left=.5in, right=.5in]{geometry}

\usepackage{amsmath,amsfonts,amssymb,epsfig,graphicx}
\usepackage{xspace,multicol,multirow,tabu,numprint,colortbl}
\usepackage{theorem}
\usepackage{color}
\usepackage{caption, float}

\newtheorem{theorem}{Theorem}[section]
\newtheorem{corollary}[theorem]{Corollary}
\newtheorem{lemma}[theorem]{Lemma}
\newtheorem{question}{Question}
\newtheorem{conjecture}{Conjecture}

\newcommand{\ds}{\displaystyle}
\newcommand{\Z}{\mathbb{Z}}
\newcommand{\V}{\mathbb{V}}
\newcommand{\mchoose}[2]{\left(\!\binom{#1}{#2}\!\right)}

\newenvironment{proof}[1][Proof]{\begin{trivlist}
\item[\hskip \labelsep {\bfseries #1}]}{\end{trivlist}}

\newcommand{\qed}{\hfill \ensuremath{\Box}}

\begin{document}
\npthousandsep{,}
\thispagestyle{empty}
\title{\textbf{Odd Vector Cycles in $\mathbb{Z}^m$}}

\author{\textbf{Gaston A. Brouwer\footnote{gaston.brouwer@mga.edu} \hspace{20pt}Jonathan Joe\footnote{jonathan.joe@mga.edu} \hspace{20pt}Matt Noble\footnote{matthew.noble@mga.edu }}  \\
Department of Mathematics and Statistics\\
Middle Georgia State University\\
Macon, GA 31206}

\date{}
\maketitle

\begin{abstract}
Given positive integers $m$ and $r$, define $C_m(r)$ to be the minimum odd number of $\mathbb{Z}^m$ vectors, each of magnitude $\sqrt{r}$, that together sum to the zero vector.  In this article, $C_m(r)$ is investigated for various assignments of $m$ and $r$.  A few previous results are combined to definitively answer the question except in the case of $m=3$ and the square-free part of $r$ being even and also containing at least one odd prime factor $x$ with $x \equiv 2 \pmod 3$.  We detail the results of a computer-assisted search to determine $C_3(r)$ for all $r < 10^6$ and then discuss parameterizations of vector cycles in $\mathbb{Z}^3$ of length five.  We close with a few conjectures and open questions.\\

\end{abstract}

\section{Introduction}

For positive integers $m,n,r$ with $n$ odd, define an \textit{odd vector cycle} to be a collection of vectors $v_1, \ldots, v_n \in \mathbb{Z}^m$, where $|v_i| = \sqrt{r}$ for each $i \in \{1, \ldots, n\}$ and $v_1 + \cdots + v_n = \mathbf{0}$.  Although it is common practice to refer to $|v|$ as being the length of the vector $v$, throughout this paper the word ``length" will be reserved for the value of $n$ -- that is, the number of vectors in the vector cycle.  We will instead use the word ``magnitude" to refer to $|v|$.  Note also that in the above definition we are not requiring the vectors $v_1, \ldots, v_n$ to be distinct, so in fact $\{v_1, \ldots, v_n\}$ may be a multiset.  The primary question we will consider is the following:

\begin{center} ``Given $m, r \in \mathbb{Z}^+$, what is the minimum $n$ such that an odd vector cycle exists?"
\end{center}

\noindent We will notate the answer to this question as $C_m(r)$, formally defined as the minimum length of an odd vector cycle of $\mathbb{Z}^m$ vectors, each of magnitude $\sqrt{r}$.  If, for a selection of $m$ and $r$, no odd vector cycle exists for any value of $n$, we will set $C_m(r) = 0$.

While investigating $C_m(r)$ for various assignments of $m$ and $r$, it is often convenient to word an observation or result in the notation and language of Euclidean distance graphs (see \cite{soifer} for an expansive history).  For $X \subset \mathbb{R}^m$ and $d > 0$, denote by $G(X, d)$ the graph with vertex set $X$ where any two vertices are adjacent if and only if they are a Euclidean distance $d$ apart.  In general, an odd vector cycle corresponds to a closed walk in the graph $G(\mathbb{Z}^m, \sqrt{r})$, however, observe that a minimum odd vector cycle for given $m,r$ corresponds to a minimum odd cycle in $G(\mathbb{Z}^m, \sqrt{r})$.  For an arbitrary graph $H$, the length of a minimum odd cycle (should one exist) is often referred to as the \textit{odd girth} of $H$.  So, should it make a reader more comfortable, one could restate our central question as asking for the odd girth of the graph $G(\mathbb{Z}^m, \sqrt{r})$.  We won't mind if you do.

The structure of this paper will be as follows.  In Section 2, we cobble together from a variety of sources preliminary results that are used to determine $C_m(r)$ for all $r$ with $m \neq 3$.  On the other hand, a complete resolution of $C_3(r)$ for all $r$ appears to be a much more daunting task.  We begin by observing that it suffices to consider only $r \equiv 2 \pmod 4$, and we then partition the positive integers congruent to 2 modulo 4 into sets $S$ and $T$.  Set $S$ consists of all positive integers congruent to 2 modulo 4 whose square-free part contains no odd prime factor congruent to 2 modulo 3, while $T$ consists of of all positive integers congruent to 2 modulo 4 whose square-free part contains at least one odd prime factor congruent to 2 modulo 3.  This partition will play a key role in our work, so just to be clear, $S = \{2, 6, 14, 18, 26, \ldots \}$ and $T = \{10, 22, 30, 34, 46, \ldots \}$.  We employ results of Ionascu \cite{ionascu} on equilateral triangles whose vertices are points of $\mathbb{Z}^3$ to show that for all $s \in S$, $C_3(s) = 3$, while for all $t \in T$, $C_3(t) \geq 5$.  To fully resolve our primary question, it remains to determine $C_3(t)$ for each $t \in T$.

In Section 3, we detail and give the results of a computer-assisted search that determines $C_3(t)$ for all $t < 10^6$.  In Section 4, we discuss the possibilities and pitfalls that arise in attempting to develop parameterizations of 5-cycles in $G(\mathbb{Z}^3, \sqrt{t})$ for $t \in T$.  It is shown that the issue of finding a 5-cycle in $G(\mathbb{Z}^3, \sqrt{t})$ is connected to the classical problem of determining whether or not $t$ can be uniquely written as a sum of three squares.  In Section 5, we present a few questions for future work.  We then give in the Appendix our calculation of $C_3(x)$ for $x \in S \cup T$ with $x < 2000$.  These small values of $x$ include all those we have found with $C_3(x) > 5$, and as well, the chart makes for a nice visual in seeing the behavior of the function $C_3(x)$.  We also note that those values $C_3(x)$ can be found in OEIS entry A309339. 

\section{Preliminaries}

 Our goal in this section is to separate those $m, r$ for which $C_m(r)$ is known from those for which it is unknown.  We begin with Lemma \ref{oddr}, an observation seen without proof in \cite{chow} and \cite{manturov} (and probably elsewhere as well).  However, if proof is required, one need only define a bipartition $(A,B)$ of $\mathbb{Z}^m$ where $A$ consists of all $\mathbb{Z}^m$ points whose coordinate entries sum to an odd integer, and $B$ consists of all $\mathbb{Z}^m$ points whose coordinate entries sum to an even integer.  For odd $r \in \mathbb{Z}^+$, any two $\mathbb{Z}^m$ points distance $\sqrt{r}$ apart must be such that one of those points is in $A$ and the other is in $B$.  It follows that no odd vector cycle of magnitude $\sqrt{r}$ exists.

\begin{lemma} \label{oddr} For any odd $r \in \mathbb{Z}^+$, $C_m(r) = 0$ for all $m$.
\end{lemma}

In \cite{abrams}, it is shown that for any distance $d$ realized between points of $\mathbb{Q}^2$, the graph $G(\mathbb{Q}^2, d)$ has chromatic number 2.  Translating this fact back into the notation of our problem, we have the following lemma.

\begin{lemma} \label{z2} For any $r \in \mathbb{Z}^+$, $C_2(r) = 0$.
\end{lemma}

We now present a stronger statement that implies Lemma \ref{z2} when taken in conjunction with Lemma \ref{oddr} above.  We include it as a matter of secondary interest.

\begin{theorem} \label{subgraph} For any graph $H$ and $r \in \mathbb{Z}^+$, $H$ is a subgraph of $G(\mathbb{Z}^2, \sqrt{r})$ if and only if $H$ is a subgraph of $G(\mathbb{Z}^2, \sqrt{2r})$.
\end{theorem}

\begin{proof} Let $G_1 = G(\mathbb{Z}^2, \sqrt{r})$, and let $G_2 = G(\mathbb{Z}^2, \sqrt{2r})$.  Let $M = \begin{bmatrix} 1 & 1 \\ 1 & -1 \end{bmatrix}$ and suppose $x_1, x_2 \in V(G_1)$ where $x_1 = (a_1, b_1)$ and $x_2 = (a_2, b_2)$ with $|x_1 - x_2| = \sqrt{z}$.  Then $Mx_1 = (a_1 + b_1, a_1 - b_1)$ and $Mx_2 = (a_2 + b_2, a_2 - b_2)$, and it follows that $|Mx_1 - Mx_2| = \sqrt{2r}$.  Thus any subgraph of $G_1$ is a subgraph of $G_2$.

Now let $H$ be a subgraph of $G_2$.  Without loss of generality, we may assume $H$ is connected and that $H$ contains the origin $(0,0)$ in its vertex set.  Note that for any $y_1, y_2 \in V(H)$ with $y_1 = (c_1, d_1)$, $y_2 = (c_2, d_2)$, and $|y_1 - y_2| = \sqrt{2r}$, we have that $(c_1 - c_2)^2 + (d_1 - d_2)^2 \equiv 0 \pmod 2$.  Since $m^2 \equiv m \pmod 2$ for any integer $m$, we have that $(c_1 + d_1) \equiv (c_2 + d_2) \pmod 2$ as well.  Putting these facts together, it follows that for any $(\alpha, \beta) \in V(H)$, $\alpha + \beta$ is even.  Again considering the transformation $M$ defined above, $M^{-1} = \begin{bmatrix} \frac12 & \frac12 \\ \frac12 & -\frac12 \end{bmatrix}$ which leaves us with $M^{-1}y_1 = (\frac{c_1 + d_1}{2}, \frac{c_1 - d_1}{2})$ and $M^{-1}y_2 = (\frac{c_2 + d_2}{2}, \frac{c_2 - d_2}{2})$ with the important observation being that $M^{-1}y_1$ and $M^{-1}y_2$ are both elements of $\mathbb{Z}^2$.  Since $|M^{-1}y_1 - M^{-1}y_2| = \sqrt{r}$, we have that $H$ is a subgraph of $G_1$.\qed
\end{proof}

As it turns out, determining $C_m(r)$ for $m \geq 4$ is a straightforward matter.

\begin{lemma} \label{k4subgraph} Let $m,r \in \mathbb{Z}^+$ with $r$ even and $m \geq 4$.  Then the complete graph $K_4$ is a subgraph of $G(\mathbb{Z}^m, \sqrt{r})$.
\end{lemma}

\begin{proof} Let $r = 2x$ for some positive integer $x$.  Lagrange's four-square theorem guarantees the existence of $x_1, \ldots, x_4 \in \mathbb{Z}$ such that $x_1^2 + \cdots + x_4^2 = x$.  Now consider the four points of $\mathbb{Z}^4$ given below.\\

\noindent $P_1 = (0,0,0,0)$\\
$P_2 = (x_1 - x_2, x_1 + x_2, x_3 - x_4, x_3 + x_4)$\\
$P_3 = (x_1 - x_3, x_2 + x_4, x_1 + x_3, -x_2 + x_4)$\\
$P_4 = (x_1 + x_4, x_2 + x_3, -x_2 + x_3, -x_1 + x_4)$\\

\noindent A quick calculation shows that these points constitute the vertices of the complete graph $K_4$ appearing as a subgraph of $G(\mathbb{Z}^4, \sqrt{r})$. Hence $G(\mathbb{Z}^m, \sqrt{r})$ has subgraph $K_4$ for any $m \geq 4$.\qed
\end{proof}

Of course, we have Lemma \ref{z4andabove} below as an immediate corollary of Lemma \ref{k4subgraph}.  However, we present Lemma \ref{k4subgraph} as it is given because it also answers a question raised by Manturov in \cite{manturov}.  There, he finds examples of Euclidean distance graphs $G(\mathbb{Z}^3, d)$ having chromatic number 3, and then later conjectures that such examples $G(\mathbb{Z}^m, d)$ do not exist for $m > 3$.  Indeed, Lemma \ref{k4subgraph} (along with the previous Lemma \ref{oddr}) guarantees that they do not.

\begin{lemma} \label{z4andabove} For $m,r \in \mathbb{Z}^+$ with $r$ even and $m \geq 4$, $C_m(r) = 3$.
\end{lemma}

We now return our attention to odd vector cycles in $\mathbb{Z}^3$.  For any integers $x_1, x_2, x_3$ satisfying $x_1^2 + x_2^2 + x_3^2 \equiv 0 \pmod 4$, we have that $x_1, x_2, x_3$ are each even.  This implies that for all $r \in \mathbb{Z}^+$ and any given graph $H$, we have $H$ being a subgraph of $G(\mathbb{Z}^3, \sqrt{r})$ if and only if it is a subgraph of $G(\mathbb{Z}^3, 2\sqrt{r})$.  Using this observation along with Lemma \ref{oddr}, to completely describe $C_3(r)$, we need only consider $r \equiv 2 \pmod 4$.  Partition those integers into sets $S$ and $T$ as described in Section 1.  In a series of papers, Ionascu analyzes equilateral triangles whose vertices are points of $\mathbb{Z}^3$.  Most notably in \cite{ionascu}, he shows that such an equilateral triangle exists of side length $\sqrt{r}$ if and only if the square-free part of $r$ is even, but contains no odd factor congruent to 2 modulo 3.  This result is essential in providing a jumping off point for our current work.  We rephrase it using our particular notation in Theorem \ref{sandt} below.

\begin{theorem} \label{sandt} For all $s \in S$, $C_3(s) = 3$ and for all $t \in T$, $C_3(t) \geq 5$.
\end{theorem}

\section{Search Results}

In this section, we present the methods and results of a computer search used to determine $C_3(t)$ for $t \in T$, where $t<10^6$. We begin by detailing a brute force approach and then describing how it can be made more efficient.

\begin{center}
\textit{The brute force approach}
\end{center}

For a given natural number $z \in \Z^+$,  define $P(z)$ to be the number of integer solutions to $z = a^2 + b^2 + c^2$ where $0 \leq a \leq b \leq c $. For example, $P(1002)=4$; the triples $\{a,b,c\}$ such that $a^2+b^2+c^2=1002$ are given below.
$$\begin{matrix}
     4  &   5  &  31\\
     4  &  19  &  25\\
     7  &  13  &  28\\
    11  &  16  &  25
\end{matrix}$$

 Note that each triple $\{a,b,c\}$ can be permuted in $3!$ ways and there are $2^3$ possible sign changes which gives rise to $3!\cdot2^3 = 48$ possible vectors in $\Z^3$ if $a$, $b$, and $c$ are all distinct (and nonzero). Since  repetitions are involved for triples like $\{0,3,17\}$ and $\{9,9,14\}$ this means that $48\cdot P(z)$ gives an upper bound for the number of vectors that can be generated. When $t=1002$ there are exactly $48\cdot P(1002)=192$ possible vectors that can be used to form a cycle. If we are using a brute force approach to find a cycle of length five, we need to generate all 5-combinations with repetitions (also known as multisubsets of size 5). In general, for a set of size $m$ there are $\mchoose{m}5 = \binom{m+5-1}5$ such combinations. With $m=192$ we obtain that the search space for 5-cycles when $t=1002$ has size \numprint{2289653184}. The maximum sizes of the search space are listed in the table below for a few values of $t\in T$.

\begin{table}[h]
\centering
{\tabulinesep=1.2mm
\begin{tabu}{|c|c|c|c|c|}
\hline
\multirow{ 2}{*}{$t$} & \multirow{ 2}{*}{$P(t)$} & \multicolumn{3}{|c|} {Maximum search space size for cycles of length}\\ \cline{3-5}
 & & $n=5$ & $n=7$ & $n=9$\\ \hline
 190 & 1 & $\mchoose{48}{5}\approx 2.6 \times 10^6$ & $\mchoose{48}{7}\approx 1.8 \times 10^{8}$ & \cellcolor{yellow} $\mchoose{48}{9}\approx 7.6 \times 10^{9}$ \\ \hline
1002 & 4 & \cellcolor{yellow} $\mchoose{192}{5}\approx 2.3 \times 10^9$ & $\mchoose{192}{7}\approx 2.1 \times 10^{12}$ & $\mchoose{192}{9}\approx 1.2 \times 10^{15}$ \\ \hline
1978 & 3 & $\mchoose{144}{5}\approx 5.5 \times 10^8$ & \cellcolor{yellow} $\mchoose{144}{7}\approx 2.9 \times 10^{11}$ & $\mchoose{144}{9}\approx 9.4 \times 10^{13}$ \\ \hline

\numprint{99994} & 49 & \cellcolor{yellow} $\mchoose{2352}{5}\approx 6.0 \times 10^{14}$ & $\mchoose{2352}{7}\approx 8.0\times 10^{19}$ & $\mchoose{2352}{9}\approx 6.2 \times 10^{24}$ \\ \hline
\numprint{999994} & 126 & \cellcolor{yellow} $\mchoose{6048}{5}\approx 6.8 \times 10^{16}$ & $\mchoose{6048}{7}\approx 5.9\times 10^{22}$ & $\mchoose{6048}{9}\approx 3.0 \times 10^{28}$ \\ \hline
\end{tabu}}
\caption{Maximum search space size of odd cycles for various $t\in T$}
\label{searchspacesize}
\end{table}

For $t=1002$, one can inspect that the vectors $$\langle 11, 16, 25 \rangle, \langle 11, 16, -25 \rangle, \langle -16, -25, -11 \rangle, \langle 25, -11, -3 \rangle, \langle -31, 4, -5 \rangle$$ form a 5-cycle and we therefore would not have to go through the entire search space to find this by a brute force approach. Depending on the density of 5-cycles in the search space this may be a significant time saver.  However, it turns out that $C_3(190)=9$ and a brute force approach would have to contend with going through the entire search spaces for 5-cycles and 7-cycles before it can hope to get lucky in the 9-cycle search space. For larger values of $t$ (say, of an order of magnitude over $10^5$), the maximum search space is simply too large even for 5-cycles to efficiently search on a regular PC.

This style of search approach can be bettered by borrowing a well-known algorithm from computer science.  At its heart, the problem of determining the length of a minimum odd cycle in the graph $G(\mathbb{Z}^3, \sqrt{t})$ is a specialized version of the classical subset sum problem where one is given a set of some number of integers and then asked, for a given value $k$, whether it is possible to select $k$ of those integers that together sum to 0.  In our case, we are of course considering vector sums, repeats are allowed for the individual selections, and we desire to run through $k = 5, 7, 9, \ldots$ until a solution is produced.  Such problems are commonly approached by means of a ``meet in the middle" style algorithm, where instead of investigating all $k$-element sums, the one conducting the search creates a list of all $\lfloor \frac{k}{2} \rfloor$-element sums and a list of all $\lceil \frac{k}{2} \rceil$-element sums.  These lists are then ordered (in our case, by the magnitude of the resulting vector sum) and compared to see if there exists $v$ on the first list with a corresponding $-v$ on the second list.  Such an algorithm is more efficient in the worst-case scenario of having to show via exhaustion that an odd vector cycle does not exist for some particular selection of $k$, however there is always the fixed cost of generating all vector sums on the two lists, and as mentioned above, for small $t$, the density of $k$-cycles in the search space may be large enough that one should instead rely on a basic brute force search and just hope to get lucky.

\begin{center}
\textit{A modified brute force approach}
\end{center}

Here we will describe what ended up for us being a more effective way of finding odd vector cycles. In Theorem \ref{t1set} we will show that 1978 is the largest $t \in T$ less than $10^6$ for which $C_3(t)>5$ (see also Table \ref{c3valuesgt5}). This result leads us to believe that if $t \in T$ and $t>1978$, then $C_3(t)=5$. If so, we can confine our modified brute force approach to finding odd cycles of length 5 only. Here is our plan of attack:

\begin{enumerate}
\item Generate a (sufficiently large) $t \in T$.
\item \label{triples} Find all $P(t)$ triples $\{a,b,c\}$ such that $t = a^2 + b^2 + c^2$, where $0 \leq a \leq b \leq c$.
\item Generate all possible vectors from these triples by permutations and/or sign changes.
\item \label{remdup} Remove any duplicates and call this set of vectors $\V(t)$:
$$ \V(t) = \{ v_1, v_2, \hdots, v_N \} , \textrm{ where } N \leq 48 \cdot P(t)$$
\item  Generate triples of vectors $\{v_i,v_j,v_k\}$ from $\V(t)$, where $1\leq i\leq j\leq k \leq N$.
\item \label{equals} Find the sum  $s=v_i+v_j+v_k$ for each triple. If one of the components of $s$ equals zero, check whether one of the following holds for any triple $\{a,b,c\}$ found in step \ref{triples}:
$s  =  \langle 2a,2b,0 \rangle $, $s  =  \langle 2a,0,2c \rangle $, or $s = \langle 0,2b,2c \rangle $
\item If one of the equalities in step \ref{equals} holds (without loss of generality, let's assume the first equality hold) we have found a 5-cycle:
$$v_i,v_j,v_k, \langle -a, -b, c \rangle, \langle -a, -b, -c \rangle$$
\end{enumerate}

The success of this method hinges on the existence of 5-cycles where two vectors are of the form $\langle -a, -b, c \rangle, \langle -a, -b, -c \rangle$ (or a variation thereof). From our search results, such  5-cycles exist for all values $t \in T$ for which $4 \times 10^4 < t < 10^6$. The following table shows the comparison of the search space sizes for two values of $t$:

\begin{table}[h]
\centering
{\tabulinesep=1.2mm
\begin{tabu}{|c|c|c|c|}
\hline
\multirow{ 2}{*}{$t$} & \multirow{ 2}{*}{$P(t)$} & \multicolumn{2}{|c|}{Maximum search space size for 5-cycles}\\ \cline{3-4}
 & & Brute Force & Modified Brute Force  \\ \hline

\numprint{99994} & 49 &  $\mchoose{2352}{5}\approx 6.0 \times 10^{14}$ & $3\cdot49\cdot \mchoose{2280}{3}\approx 2.9\times 10^{11}$ \\ \hline
\numprint{499998} & 70 & $\mchoose{3360}{5}\approx 3.6 \times 10^{15}$ & $3\cdot70\cdot \mchoose{3360}{3}\approx 1.3\times 10^{12}$ \\ \hline
\numprint{999994} & 126 &  $\mchoose{6048}{5}\approx 6.8 \times 10^{16}$ & $3\cdot 126 \cdot \mchoose{6048}{3}\approx 1.4\times 10^{13}$  \\ \hline
\end{tabu}}
\caption{Comparison of maximum search space sizes}
\label{searchspacecomp}
\end{table}

 When looking at these numbers we have to keep in mind that in our modified brute force approach the comparison of the sum $s=v_i+v_j+v_k$ to $\langle 2a,2b,0 \rangle $, $ \langle 2a,0,2c \rangle $, or $ \langle 0,2b,2c \rangle $ is only triggered when one of the components of $s$ equals zero. The factor $3\cdot P(t)$ is therefore generously overestimating the true size of the search space in this case. However, even when we do not account for that, we see a significant reduction in the size of the search space. Note that for the $t$-values listed in this table, removing duplicate vectors (step \ref{remdup} in our modified approach), only results in a (small) reduction of the search space for $t=\numprint{99994}$ ($N=2280$ versus $N=2352$).

\begin{center}
\textit{Results}
\end{center}

Define sets $T_0 = \{t \in T: t < 10^6\}$ and $T_1 = \{t \in T_0: C_3(t) > 5\}$. The computer code for the modified brute force approach ran on approximately 40 PC's in a computer lab with each PC being tasked with finding 5-cycles for a different subset of $T_0$. We estimate that the total computing time for each PC did not exceed two weeks.

\begin{theorem}\label{t1set}
The set $T_1$ is given by $$T_1 = \{22, 58, 70, 82, 142, 190, 298, 330, 358, 382, 478, 658, 742, 862, 1222, 1978\}\text{.}$$
\end{theorem}
\begin{proof}
An exhaustive search over all $T_0$ found the $t \in T_0$ for which $C_3(t)>5$.  They are given in Table 3.
\begin{table}[h]
\centering
\begin{tabu}{||c|c||c|c||c|c||c|c||}
\hline
$t$ & $C_3(t)$  & $t$ & $C_3(t)$ & $t$ & $C_3(t)$ & $t$ & $C_3(t)$ \\
\hline
22 & 9 & 142 & 7 & 358 & 7 &742 & 7\\
58 & 11 & 190 & 9 &382 & 7 & 862 & 7 \\
70 & 7 &298 &7 &478 & 7 & 1222 &7 \\
82 & 7 & 330 & 7 & 658 & 7 & 1978 & 7\\
\hline
\end{tabu}
\caption{Values of $C_3(t)$ for $t \in T_1$}
\label{c3valuesgt5}
\end{table}

\end{proof}


Theorem \ref{t1set} leads us to make the following conjecture.

\begin{conjecture} \label{z3conjecture} For all $t \in T \setminus T_1$, $C_3(t) = 5$.
\end{conjecture}

For readers who are unsatisfied by the search methods described above, in Theorems \ref{magnitude22} and \ref{magnitude82} we formally show for two values of $t\in T_0$ that $C_3(t)>5$.  We remark that, for the most part, such proofs are not particularly hard to construct, but the arguments can become quite tedious as $P(t)$ grows larger.

\begin{theorem} \label{magnitude22} $C_3(22) = 9$.
\end{theorem}

\begin{proof} Ignoring permutations and negatives of coordinate entries, it stands that $\langle 2,3,3 \rangle$ is the only $\mathbb{Z}^3$ vector of magnitude $\sqrt{22}$.  Let $v_1, \ldots, v_n$ constitute an odd vector cycle of the appropriate magnitude, and designate by $\mathcal{C}_x$ the multiset of $x$-coordinate entries of these vectors.  Similarly define $\mathcal{C}_y$ and $\mathcal{C}_z$.  For each of these $\mathcal{C}_i$ to have elements that together sum to 0, we must have in each an even number of odd values (and by extension, an odd number of even values).  Also observe that in each $\mathcal{C}_i$, we cannot have only one of its $n$ elements being equal to $\pm 2$, seeing as that would make the sum of all the elements of that $\mathcal{C}_i$ not congruent to $0 \pmod 3$.  Taking these observations together, we have that each $\mathcal{C}_i$ must have three of its elements from the set $\{2,-2\}$, and it follows that $C_3(22) \geq 9$.  To show that $C_3(22) = 9$, we just need to note that the following nine vectors sum to the zero vector: $\langle 2, -3, -3 \rangle, \langle 2, -3, -3 \rangle, \langle 2, -3, -3 \rangle, \langle -3, 2, -3 \rangle, \langle -3, 2, 3 \rangle, \langle -3, 2, 3 \rangle, \langle -3, -3, 2 \rangle, \langle 3, 3, 2 \rangle, \langle 3, 3, 2 \rangle$.\qed
\end{proof}

\begin{theorem} \label{magnitude82} $C_3(82) = 7$.
\end{theorem}

\begin{proof}   Observe that, ignoring permutations and negatives of coordinate entries, the vectors $\langle 9, 1, 0 \rangle$ and $\langle 8, 3, 3 \rangle$ are the only two in $\mathbb{Z}^3$ of magnitude $\sqrt{82}$.  Assume to the contrary that $C_3(82) = 5$, and consider an odd vector cycle of length five.  Define $\mathcal{C}_x$, $\mathcal{C}_y$, and $\mathcal{C}_z$ as is done in the proof of Theorem \ref{magnitude22}.  Each of the vectors $\langle 9, 1, 0 \rangle$ and $\langle 8, 3, 3 \rangle$ contain exactly one entry that is not congruent to $0 \pmod 3$, so in this supposed vector cycle, there will be a total of five integers not congruent to $0 \pmod 3$ distributed across the $\mathcal{C}_x$, $\mathcal{C}_y$, $\mathcal{C}_z$.  Let $\mathcal{V}$ denote the multiset of these five integers.  In order for the respective sums of the entries of $\mathcal{C}_x$, $\mathcal{C}_y$, $\mathcal{C}_z$ to each be congruent to $0 \pmod 3$, we must have either one of these $\mathcal{C}_i$ containing all five of the elements of $\mathcal{V}$, or one of the $\mathcal{C}_i$ containing three elements of $\mathcal{V}$ and another containing the other two elements of $\mathcal{V}$.  However, we can discount the former possibility as five integers taken from $\{\pm 1, \pm 8\}$ cannot sum to zero.

Without loss of generality, assume $\mathcal{C}_x$ contains three elements of $\mathcal{V}$ and $\mathcal{C}_y$ contains the other two elements of $\mathcal{V}$.  It doesn't take too much to then see that the vector cycle must take one of the following forms.

\begin{multicols}{3}

\setlength{\columnseprule}{0.4pt}

\begin{center} $\langle 8, \, , \,\rangle$ \\ $\langle -1, \, , \, \rangle$ \\ $\langle -1, \, , \, \rangle$ \\ $\langle -9, 1, \, \rangle$ \\ $\langle 3, 8, \, \rangle$\\
\end{center}

\begin{center} $\langle 8, \, , \, \rangle$ \\ $\langle -1, \, , \, \rangle$ \\ $\langle -1, \, , \, \rangle$ \\ $\langle -3, 8, \, \rangle$ \\ $\langle -3, -8, \, \rangle$ \\
\end{center}

\begin{center} $\langle 1, \, , \, \rangle$ \\ $\langle 1, \, , \, \rangle$ \\ $\langle 1, \, , \, \rangle$ \\ $\langle -3, 8, \, \rangle$ \\ $\langle 0, 1, \, \rangle$ \\
\end{center}
\end{multicols}

However, none of these three forms can be completed into a vector cycle, with probably the quickest way of seeing this being to note the impossibility in each case of both $\mathcal{C}_y$ and $\mathcal{C}_z$ having entries whose sum is congruent to $0 \pmod 9$.  Finally, we have the seven vectors $\langle -9, 0, -1 \rangle, \langle -1, -9, 0 \rangle, \langle 8, -3, 3 \rangle, \langle -9, 0, 1 \rangle, \langle 3, 8, 3 \rangle,\\ \langle 8, 3, 3 \rangle$, and $\langle 0, 1, -9 \rangle$ constituting an odd vector cycle and confirming that $C_3(82) = 7$.\qed
\end{proof}

\section{Parameterizations}

We begin this section by returning for a moment to \cite{ionascu}.  There, Ionascu observes that for any $a,b \in \mathbb{Z}$, the vectors $\langle -a, -b, a+b \rangle$, $\langle -b, a+b, -a \rangle$, $\langle a+b, -a, -b \rangle$ sum to $\langle 0,0,0 \rangle$.  Each of those vectors has length $\sqrt{2a^2 + 2ab + 2b^2}$, so for any integer $z$ that can be represented by the quadratic form $2a^2 + 2ab + 2b^2$, we have the existence of an equilateral triangle in $\mathbb{Z}^3$ of side length $\sqrt{z}$.  Ionascu then employs a classical result of Euler to show that the integers that can be represented by this form are exactly those of our set $S$, so for our means, we have $C_3(s) = 3$ for each $s \in S$.  Furthermore, he gives an alternate characterization of the set $S$ by noting that any $z$ congruent to 2 modulo 4 is in $S$ if and only if $z$ can be written as a sum of three integer squares $a^2 + b^2 + c^2$ where $a + b = c$.    With these ideas in mind, one may now ask if it is possible to implement a similar technique to parameterize 5-cycles in $G(\mathbb{Z}^3, \sqrt{t})$ for some $t \in T$.  In short, we have found that it is possible, but only to some extent as difficulties abound.  We will make note of this by illustrating the connection of our problem to a well-studied problem of classical number theory, that of deciding which integers have a unique representation as a sum of three squares.

Volume after volume has been written on representations of integers as sums of squares, and for a historical perspective, we suggest \cite{grosswald} or \cite{moreno}.  We of course will not go as in-depth here, but a little background is needed. Similarly given in Section 3, for $z \in \mathbb{Z}^+$, define $P(z)$ to be the number of integer solutions to $z = a^2 + b^2 + c^2$ where $0 \leq a \leq b \leq c$.  For the same reason given in Section 2 that $C_3(z) = C_3(4z)$, we have $P(z) = P(4z)$ as well, so it makes sense to only consider $z \equiv 1,2,3,5,6, \text{ or } 7 \pmod 8$.  A classical result due to Legendre shows that, for those values of $z$, $P(z) \geq 1$ if and only if $z \not \equiv 7 \pmod 8$.  Furthermore, Legendre proves that each such $z$ has a representation $z = a^2 + b^2 + c^2$ where $\gcd(a,b,c) = 1$.  Note that this implies that for any $z$ with $P(z) \geq 1$ and odd integer $n$, we have $P(zn^2) \geq 2$.  This fact will come into play in the proof of Theorem \ref{conditional} below.

Denote by $K$ the set of all such integers $z$ such that $P(z) = 1$.  There are thirty-two integers known to be elements of $K$ with the largest of these being 427.  Even though this problem has been considered for hundreds of years, proof that those thirty-two integers are in fact the only elements of $K$ is a relatively recent development.  In \cite{bateman}, Bateman and Grosswald detail how this can be shown by determining all imaginary quadratic fields of class number 4.  This is finally accomplished in a 1992 article \cite{arno} by Arno, however his proof required a massive computer assist.

Our goal will now be to prove that if Conjecture 1 is resolved in the affirmative, then even without Arno's result, it would give proof that the only $t \in T$ satisfying $P(t) = 1$ are those that appear on the previously mentioned list of thirty-two integers.  To do this, we will call upon a result found in \cite{borosh2} concerning solutions of systems of linear Diophantine equations.  It is given as Theorem \ref{matrices} below.

\begin{theorem} \label{matrices} Let $A$ be an $n \times n$ matrix of rank $r = n - 1$.  If $Ax = B$ has a non-trivial, non-negative solution, then it has such a solution with $\max x_i \leq M$, where $M$ is the maximum of the values of all the minors of order $r$ of $(A|B)$.
\end{theorem}

From the previous section, recall $T_0 = \{t \in T: t < 10^6\}$.


\begin{theorem} \label{conditional} Let $t \in T \setminus T_0$ and suppose $P(t) = 1$.  Then $C_3(t) > 5$.
\end{theorem}

\begin{proof} Let $t \in T \setminus T_0$ and suppose $P(t) = 1$, which by Legendre's result guarantees that $t$ is square-free.  Let $\langle a,b,c \rangle \in \mathbb{Z}^3$ where $a^2 + b^2 + c^2 = t$.  As $t \equiv 2 \pmod 4$, we may assume $a$ even and $b,c$ odd.  Consider a collection of $\mathbb{Z}^3$ vectors $v_1, \ldots, v_5$, each of magnitude $\sqrt{t}$, that satisfy $v_1 + \cdots + v_5 = \mathbf{0}$.  Let $\mathcal{C}_x$ be the collection (potentially a multiset) of the $x$-component entries of these vectors, and similarly define $\mathcal{C}_y$ and $\mathcal{C}_z$.  Note that in order for the elements of each of $\mathcal{C}_x$, $\mathcal{C}_y$, $\mathcal{C}_z$ to sum to zero, we must have each of the collections having an odd number of their entries from the set $\{a, -a\}$.  Without loss of generality we may freely assume that $\mathcal{C}_x$ has three of its entries equal to $\pm a$, while $\mathcal{C}_y$ and $\mathcal{C}_z$ have one each.

For $i \in \{x,y,z\}$, let $\alpha_ia$ be the sum of the entries in $\mathcal{C}_i$ equal to $\pm a$.  Similarly, let $\beta_ib$ be the sum of the entries in $\mathcal{C}_i$ equal to $\pm b$, and let $\gamma_ic$ be the sum of the entries in $\mathcal{C}_i$ equal to $\pm c$.  This gives rise to the system of linear Diophantine equations below, where each coefficient is in the interval $[-4,4]$.

$$\alpha_xa + \beta_xb + \gamma_xc = 0$$
$$\alpha_ya + \beta_yb + \gamma_yc = 0$$
$$\alpha_za + \beta_zb + \gamma_zc = 0$$

For $i \in \{x,y,z\}$, we have $|\mathcal{C}_i| = 5$, so it stands that $|\alpha_i| + |\beta_i| + |\gamma_i| \leq 5$.   However, we cannot have $|\alpha_i| = |\beta_i| = |\gamma_i| = 1$, as that would imply that $t$ can be written as a sum of three squares, say $l^2 + m^2 + n^2$, where $l + m = n$.  This would in turn give the contradiction of $t \in S$.  Since there are five component entries in $v_1, \ldots, v_5$ equal to each of $\pm b$ and $\pm c$, it cannot be the case that each of the $\beta_i$ or each of the $\gamma_i$ are equal to 0.  Also, we must have the sums $|\beta_1| + |\beta_2| + |\beta_3|$ and  $|\gamma_1| + |\gamma_2| + |\gamma_3|$ each less than or equal to 5.

Letting $A = \begin{bmatrix} \alpha_x & \beta_x & \gamma_x\\\alpha_y & \beta_y & \gamma_y\\\alpha_z & \beta_z & \gamma_z \end{bmatrix}$, we now look at the possible ranks of $A$.  If rank$(A) = 3$, then the system above has its only solution being $a=b=c=0$, so this case is of no interest.  If rank$(A) = 1$, consider $u = \langle |\alpha_x|, |\alpha_y|, |\alpha_z| \rangle$.  Our previous observations indicate that $|\alpha_x| \in \{1,3\}$ and $|\alpha_y| + |\alpha_z| \leq 2$ with $\alpha_y$ and $\alpha_z$ both being non-zero.  This leaves only the possibilities of $u = \langle 1, 1, 1 \rangle$ which, as mentioned, gives $t \in S$, and $u = \langle 3, 1, 1 \rangle$ which would result in the impossibility of both the second and third rows of $A$ being scalar multiples of the first row and also $A$ having its entries all being integers.

We may now assume rank$(A) = 2$.  This indicates that we can solve the previous system of equations in terms of one of the variables, say $b$.  Write a solution vector $x = \begin{bmatrix}  pb\\b\\qb \end{bmatrix}$ for rationals $p, q$, and apply Theorem \ref{matrices}.  Since $A$ has each of its entries in the interval $[-4,4]$, we have an immediate upper bound of 32 on the value $M$ of a minor of order 2.  So there exists a $b$ giving an integer solution to the system in which $t = (pb)^2 + b^2 + (qb)^2 \leq 3(32)^2 = 3072$, a number which falls well within the bounds of $T_0$, contradicting $t \in T \setminus T_0$.  If a larger $b$ is chosen, we would have $t$ not being square-free.  Legendre's result then guarantees that $P(t) \geq 2$, contradicting the assumption that $P(t) = 1$, and completing the proof.\qed
\end{proof}

We give the contrapositive of Theorem \ref{conditional} as a corollary below, where again, $P(t)$ indicates the number of integer solutions to $t = a^2 + b^2 + c^2$ with $0 \leq a \leq b \leq c$.

\begin{corollary} Let $t \in T \setminus T_0$.  If $C_3(t) = 5$, then $P(t) \geq 2$.
\end{corollary}

Considering the long history of the problem of determining all $z$ with $P(z) = 1$, and the incredible ardor required for its eventual solution, even restricting our attention to finding all $t \in T$ with $P(t) = 1$ appears a steep task.  This in turn seems to indicate that a positive resolution of Conjecture 1 may prove to be quite difficult indeed.

We now restart from scratch in our attempt to parameterize 5-cycles in $G(\mathbb{Z}^3, \sqrt{t})$, only this time we allow ourselves more flexibility.  Instead of beginning with a single $\langle a,b,c \rangle$, we use multiple vectors of magnitude $\sqrt{t}$ (along with all vectors created by permuting and negating the entries of the originals) to construct our parameterization.  Here, we find success.  It turns out that such parameterizations exist, and we display two of those that we found below.

\begin{multicols}{2}

\setlength{\columnseprule}{0.4pt}

\begin{center} \textbf{Parameterization 1} \\[10pt] $\langle -2x-y,-x-y,-x+2y \rangle$ \\ $\langle -2x-y,-x-y,-x+2y \rangle$ \\ $\langle x+y,2x+y,x-2y \rangle$ \\ $\langle x+2y,-x-y,2x-y \rangle$ \\ $\langle 2x-y,x+2y,-x-y \rangle$\\
\end{center}

\begin{center} \textbf{Parameterization 2} \\[10pt] $\langle 7x + 10y, 7x + y, 3y \rangle$ \\ $\langle 7x + 10y, 7x + y, 3y \rangle$ \\ $\langle -7x - 6y, 7y, 7x + 5y \rangle$ \\ $\langle -3x - 9y, -5x - 2y, -8x - 5y \rangle$ \\ $\langle -4x - 5y, -9x - 7y, x - 6y \rangle$ \\
\end{center}
\end{multicols}

Each vector of Parameterization 1 has magnitude $\sqrt{6x^2 + 2xy + 6y^2}$ and each of Parameterization 2 has magnitude $\sqrt{98x^2 + 154xy + 110y^2}$.  In general, we will say that such a parameterization \textit{represents} $t \in T$ if there exist integers $x,y$ which result in the vectors of the parameterization having magnitude $\sqrt{t}$.  Should one want to extend Theorem \ref{t1set} and determine $C_3(t)$ for a collection of $t > 10^6$, this idea of parameterization may present a significant time saver.  As an example, suppose $T_2 = \{t \in T: 10^6 < t < 10^7\}$ and let $L$ be the set of all integers representable by the binary quadratic form $6x^2 + 2xy + 6y^2$ given above.  For all $t \in T_2 \cap L$, we are guaranteed that $C_3(t) = 5$ and no brute force search for a 5-cycle of magnitude $\sqrt{t}$ is needed.

Unfortunately, no finite number of these parameterizations can be used to generate 5-cycles in $G(\mathbb{Z}^3, \sqrt{t})$ for all $t \in T$.  We will show this in Theorem \ref{cannot} to follow.  Its proof will involve a number of terms and elementary techniques from classical number theory, and for a refresher on these concepts, a reader could consult virtually any introductory text, for example \cite{nagell}.

\begin{theorem} \label{cannot} Let $\mathcal{P}$ be a finite collection of parameterizations of vector cycles of length 5 in $\mathbb{Z}^3$.  Then there exists some $t \in T$ that is not represented by any parameterization in $\mathcal{P}$.
\end{theorem}

\begin{proof} Let $\mathcal{P} = \{P_1, \ldots, P_n\}$ where for $i \in \{1, \ldots, n\}$, the five vectors of $P_i$ have the square of their lengths being equal to $a_ix^2 + b_ixy + c_iy^2$.  Set $F_i = a_ix^2 + b_ixy + c_iy^2$.  As described in say, Chapter 6 of \cite{nagell}, we can perform linear transformations to change each $F_i$ into a corresponding $F_i' = a_i'x^2 + b_i'y^2$ for integers $a_i'$, $b_i'$ where an integer being represented by $F_i$ over $\mathbb{Z}$ implies that it is also represented by $F_i'$ over $\mathbb{Q}$.  With this in mind, our goal is to now show the existence of some $t \in T$ such that for each $i \in \{1, \ldots, n\}$, the equation $a_i'x^2 + b_i'y^2 = t$ is not solvable in rationals.

We may assume each $a_i'$ and $b_i'$ is square-free.  As well, we must have $a_i', b_i'$ positive.  Furthermore, we must have each $a_i', b_i'$ congruent to $2 \pmod 4$ or else $F_i'$ represents an odd integer in contradiction to Lemma \ref{oddr}.  Writing $a_i' = 2\alpha_i$, $b_i' = 2\beta_i$, and $t = 2t_0$, we simplify to arrive at the equation $\alpha_ix^2 + \beta_iy^2 = t_0$.  Letting $x = \frac{u}{w}$ and $y = \frac{v}{w}$, we move to homogeneous coordinates and consider Equation \ref{eq1} below.

\begin{equation}
\alpha_iu^2 + \beta_iv^2 - t_0w^2 = 0
\label{eq1}
\end{equation}

We intend to construct $t$ of the form $t = 2p$ where $p \equiv 2 \pmod 3$ is prime.  By a well-known necessary condition for the solvability of Diophantine equations of the form of Equation \ref{eq1} above, for any prime $p$ dividing $t_0$ where $p \nmid \alpha_i\beta_i$, we must have $-\alpha_i\beta_i$ being a quadratic residue of $p$.  Let $r = \max\{\alpha_i\beta_i : i = 1, \ldots, n\}$.  For each $j \in \{1, \ldots, r\}$, we can select a distinct prime $p_j$ such that $-j$ is a quadratic non-residue of $p_j$.  By the Chinese Remainder Theorem, there exists $m \in \mathbb{Z}^+$ satisfying the system of linear congruences $m \equiv -j \pmod {p_j}$ for all $j \in \{1, \ldots, r\}$.  For any integer $k$, we have $m + k(p_1 \cdots p_r)$ satisfying the system as well.  By Dirichlet's Theorem concerning primes in arithmetic progressions, the sequence $m + k(p_1 \cdots p_r)$ for $k = 1, 2, \ldots$ contains a prime, which we will designate as $p$.  We now employ a few basic facts about Legendre symbols to guarantee that $p \equiv 2 \pmod 3$.  Having $(\frac{-1}{p}) = -1$ implies $p \equiv 3 \pmod 4$.  As well, $(\frac{-3}{p}) = -1$ gives $(\frac{-1}{p})(\frac{3}{p}) = -1$ which means $(\frac{3}{p}) = 1$.  Putting these facts together, we have $(\frac{p}{3}) = -1$ which gives the desired $p \equiv 2 \pmod 3$.

Since the constructed $p$ is larger than any divisor of $\alpha_i$, $\beta_i$ for any $i \in \{1, \ldots, n\}$, we have $t = 2p$ not represented by any of the parameterizations $P_i$.  This completes the proof.\qed
\end{proof}

\section{Further Work}

In this section we highlight a few questions that appear to be fertile ground for continuing work.  The most immediate is Conjecture \ref{z3conjecture} which is previously given in Section 3.  However, we may also posit a similar claim concerning the rational space $\mathbb{Q}^3$.

\begin{conjecture} \label{q3conjecture} For all $t \in T$, the graph $G(\mathbb{Q}^3, \sqrt{t})$ has odd girth 5.
\end{conjecture}

One need only think about scaling by a factor of $q$ to see that that graphs $G(\mathbb{Q}^3, \sqrt{t})$ and $G(\mathbb{Q}^3, \sqrt{q^2t})$ are isomorphic for any $q \in \mathbb{Q}^+$.  If some $G(\mathbb{Q}^3, \sqrt{t})$ contained a triangle, we could scale by an appropriate integer $z$ so that the vertices of that triangle were mapped to points of $\mathbb{Z}^3$.  As $t \in T$ implies $z^2t \in T$ as well, this would contradict Theorem \ref{sandt}.  Thus $G(\mathbb{Q}^3, \sqrt{t})$ has odd girth 5 or larger.  If the above conjecture was false, we must have some $t \in T$ such that for \textit{every} positive integer $n$, $C_3(n^2t) \geq 7$.  This would fly in the face of Conjecture 1, and we remark that even for those $t \in T_1$ we have, for example, $C_3(9t) = 5$.  So barring some unforeseen conspiracy among the rational numbers, it seems almost certain that Conjecture 2 is true.  That said, we rather sheepishly admit that we were unable to prove it to be so.

\begin{question} \label{characterize} What geometric characterizations are there, if any, that can be used to describe parameterizations of 5-cycles in $G(\mathbb{Z}^3, \sqrt{t})$?
\end{question}

Again, we return to \cite{ionascu}.  There, Ionascu begins his study of equilateral triangles in $\mathbb{Z}^3$ by determining exactly the planes in which those triangles lie.  He proves that a plane $P$ contains an equilateral triangle whose vertices are points of $\mathbb{Z}^3$ if and only if $P$ has a normal vector of the form $\langle a,b,c \rangle$ where $a^2 + b^2 + c^2 = 3d^2$ for integers $a,b,c,d$.  From this initial statement, he develops a full characterization of all equilateral triangles in $\mathbb{Z}^3$.  We are wondering if anything along those lines can be done with 5-cycles.  Moreover, it may be interesting to investigate under what conditions an odd vector cycle exists with each vector lying in the same plane.

\begin{question} \label{plane} For $t \in T$, let $f(t)$ denote the minimum number of $\mathbb{Z}^3$ vectors of magnitude $\sqrt{t}$ that each lie in the same plane and together sum to the zero vector.  What can be said about $f(t)$?  When does $f(t)$ exist?
\end{question}

As a toy example, observe that $f(10)$ does not exist.  To see this, suppose to the contrary that $\mathcal{V}$ is an odd vector cycle of magnitude $\sqrt{10}$ with the property that all vectors of $\mathcal{V}$ lie in the same plane.  Any vector of $\mathcal{V}$ must be formed by permuting the entries of the vector $\langle \pm 3, \pm 1, 0 \rangle$.  Define $\mathcal{C}_x$, $\mathcal{C}_y$, and $\mathcal{C}_z$ as is done in the proof of Theorem \ref{conditional}, and note that one of $\mathcal{C}_x$, $\mathcal{C}_y$, and $\mathcal{C}_z$ must be a multiset with more than one of its entries being 0.  Without loss of generality, assume $\langle a_1, b_1, 0 \rangle, \langle a_2, b_2, 0 \rangle \in \mathcal{C}_z$.  The cross product of these two vectors is of the form $\langle 0, 0, r \rangle$ for some integer $r$, and it follows that an odd cycle exists in the graph $G(\mathbb{Z}^2, \sqrt{10})$.  This contradicts Lemma \ref{z2}.

We conclude with a remark on the distribution of the sets $S$ and $T$.  For an integer $n \geq 2$, define $S_n = \{s: s \in S \text{ and } s \leq n\}$.  Similarly, let $T_n = \{t: t \in T \text{ and } t \leq n\}$.  Define a function $g(n) = \frac{|T_n|}{|S_n \cup T_n|}$.  The following chart displays $|T_n|$ and $g(n)$ for a few values of $n$.

\begin{table}[h]
\centering
\begin{tabu}{||c|c|c||}
\hline
$n$ & $|T_n|$  & $g(n)$ \\
\hline
$2000$ & $303$ & $.606$\\
\hline
$10^5$ & 17,414 & $\approx.6966$\\
\hline
$10^6$ & 181,707 & $\approx.7268$\\
\hline
$10^7$ & 1,873,768 & $\approx.7495$\\
\hline
$10^8$ & 19,181,930 & $\approx.7632$\\
\hline
$10^9$ & 195,425,213 & $\approx .7817$\\
\hline
\end{tabu}
\caption{$|T_n|$ and $g(n)$ for various $n$}
\label{gfunction}
\end{table}

The function $g(n)$ grows quite slowly, but it is the case that $\ds \lim_{n\to\infty} g(n) = 1$.  This is due to a result of Bernays \cite{bernays}, which extends the classic observation due to Landau that the proportion of integers representable as a sum of two squares is asymptotically zero.  Consider a binary quadratic form $ax^2 + bxy + cy^2$ with non-square discriminant $d = b^2 - 4ac$, and let $L_n$ be the set of positive integers that are less than or equal to $n$ and representable by that form.  Bernays shows that $\ds \lim_{n\to\infty} \frac{|L_n|}{n} = 0$.  As shown in \cite{ionascu} (and previously mentioned in Section 4), each $s \in S$ is representable by the binary quadratic form $2x^2 + 2xy + 2y^2$ which has discriminant $-12$, and we have that $S$ has density 0 in the set of positive integers.  As $S \cup T$ is just the set of positive integers congruent to 2 modulo 4, its density is $\frac14$, and it then follows that the density of $T$ alone is $\frac14$ as well.  So, if Conjecture \ref{z3conjecture} is true, the the chart found in the appendix could be extended indefinitely, and in doing so, one would encounter nothing but 3's and 5's.  Asymptotically, one would encounter nothing but 5's!

\section*{Acknowledgments}

The authors thank Sarosh Adenwalla and Evan O'Dorney, whose comments improved the presentation of this paper and even better, allowed the authors to avoid looking foolish with a few of their claims.


\pagebreak

\section*{Appendix}

The chart below gives $C_3(n)$ for all positive integers $n \equiv 2 \pmod 4$ with $n < 2000$.  As previously described in Section 3, this range includes all $n \equiv 2 \pmod 4$ that we have found where $C_3(n) > 5$.

\begin{table}[h]
\centering
\begin{tabular}{||c|c||c|c||c|c||c|c||c|c||}
\hline
 $n$ & $C_3(n)$ & $n$ & $C_3(n)$ & $n$ & $C_3(n)$ & $n$ & $C_3(n)$ & $n$ & $C_3(n)$  \\
 \hline
 2 &  3 &  102 &  \cellcolor{yellow}5 &  202 &  \cellcolor{yellow}5 &  302 &  3 &  402 &  3 \\
 \hline
 6 &  3 &  106 &  \cellcolor{yellow}5 &  206 &  3 &  306 &  \cellcolor{yellow}5 &  406 &  \cellcolor{yellow}5 \\
 \hline
 10 &  \cellcolor{yellow}5 &  110 &  \cellcolor{yellow}5 &  210 &  \cellcolor{yellow}5 &  310 &  \cellcolor{yellow}5 &  410 &  \cellcolor{yellow}5 \\
 \hline
 14 &  3 &  114 &  3 &  214 &  \cellcolor{yellow}5 &  314 &  3 &  414 &  \cellcolor{yellow}5 \\
 \hline
 18 &  3 &  118 &  \cellcolor{yellow}5 &  218 &  3 &  318 &  \cellcolor{yellow}5 &  418 &  \cellcolor{yellow}5 \\
 \hline
 22 &  \cellcolor{red}9 &  122 &  3 &  222 &  3 &  322 &  \cellcolor{yellow}5 &  422 &  3 \\
 \hline
 26 &  3 &  126 &  3 &  226 &  \cellcolor{yellow}5 &  326 &  3 &  426 &  \cellcolor{yellow}5 \\
 \hline
 30 &  \cellcolor{yellow}5 &  130 &  \cellcolor{yellow}5 &  230 &  \cellcolor{yellow}5 &  330 &  \cellcolor{red}7 &  430 &  \cellcolor{yellow}5 \\
 \hline
 34 &  \cellcolor{yellow}5 &  134 &  3 &  234 &  3 &  334 &  \cellcolor{yellow}5 &  434 &  3 \\
 \hline
 38 &  3 &  138 &  \cellcolor{yellow}5 &  238 &  \cellcolor{yellow}5 &  338 &  3 &  438 &  3 \\
 \hline
 42 &  3 &  142 &  \cellcolor{red}7 &  242 &  3 &  342 &  3 &  442 &  \cellcolor{yellow}5 \\
 \hline
 46 &  \cellcolor{yellow}5 &  146 &  3 &  246 &  \cellcolor{yellow}5 &  346 &  \cellcolor{yellow}5 &  446 &  3 \\
 \hline
 50 &  3 &  150 &  3 &  250 &  \cellcolor{yellow}5 &  350 &  3 &  450 &  3 \\
 \hline
 54 &  3 &  154 &  \cellcolor{yellow}5 &  254 &  3 &  354 &  \cellcolor{yellow}5 &  454 &  \cellcolor{yellow}5 \\
 \hline
 58 &  \cellcolor{red}11 &  158 &  3 &  258 &  3 &  358 &  \cellcolor{red}7 &  458 &  3 \\
 \hline
 62 &  3 &  162 &  3 &  262 &  \cellcolor{yellow}5 &  362 &  3 &  462 &  \cellcolor{yellow}5 \\
 \hline
 66 &  \cellcolor{yellow}5 &  166 &  \cellcolor{yellow}5 &  266 &  3 &  366 &  3 &  466 &  \cellcolor{yellow}5 \\
 \hline
 70 &  \cellcolor{red}7 &  170 &  \cellcolor{yellow}5 &  270 &  \cellcolor{yellow}5 &  370 &  \cellcolor{yellow}5 &  470 &  \cellcolor{yellow}5 \\
 \hline
 74 &  3 &  174 &  \cellcolor{yellow}5 &  274 &  \cellcolor{yellow}5 &  374 &  \cellcolor{yellow}5 &  474 &  3 \\
 \hline
 78 &  3 &  178 &  \cellcolor{yellow}5 &  278 &  3 &  378 &  3 &  478 &  \cellcolor{red}7 \\
 \hline
 82 &  \cellcolor{red}7 &  182 &  3 &  282 &  \cellcolor{yellow}5 &  382 &  \cellcolor{red}7 &  482 &  3 \\
 \hline
 86 &  3 &  186 &  3 &  286 &  \cellcolor{yellow}5 &  386 &  3 &  486 &  3 \\
 \hline
 90 &  \cellcolor{yellow}5 &  190 &  \cellcolor{red}9 &  290 &  \cellcolor{yellow}5 &  390 &  \cellcolor{yellow}5 &  490 &  \cellcolor{yellow}5 \\
 \hline
 94 &  \cellcolor{yellow}5 &  194 &  3 &  294 &  3 &  394 &  \cellcolor{yellow}5 &  494 &  3 \\
 \hline
 98 &  3 &  198 &  \cellcolor{yellow}5 &  298 &  \cellcolor{red}7 &  398 &  3 &  498 &  \cellcolor{yellow}5 \\
 \hline

\end{tabular}
\end{table}

\pagebreak

\begin{table}[h]
\centering
\begin{tabular}{||c|c||c|c||c|c||c|c||c|c||}
\hline
$n$ & $C_3(n)$ & $n$ & $C_3(n)$ & $n$ & $C_3(n)$ & $n$ & $C_3(n)$ & $n$ & $C_3(n)$  \\
 \hline
 502 &  \cellcolor{yellow}5 &  602 &  3 &  702 &  3 &  802 &  \cellcolor{yellow}5 &  902 &  \cellcolor{yellow}5 \\
 \hline
 506 &  \cellcolor{yellow}5 &  606 &  \cellcolor{yellow}5 &  706 &  \cellcolor{yellow}5 &  806 &  3 &  906 &  3 \\
 \hline
 510 &  \cellcolor{yellow}5 &  610 &  \cellcolor{yellow}5 &  710 &  \cellcolor{yellow}5 &  810 &  \cellcolor{yellow}5 &  910 &  \cellcolor{yellow}5 \\
 \hline
 514 &  \cellcolor{yellow}5 &  614 &  3 &  714 &  \cellcolor{yellow}5 &  814 &  \cellcolor{yellow}5 &  914 &  3 \\
 \hline
 518 &  3 &  618 &  3 &  718 &  \cellcolor{yellow}5 &  818 &  3 &  918 &  \cellcolor{yellow}5 \\
 \hline
 522 &  \cellcolor{yellow}5 &  622 &  \cellcolor{yellow}5 &  722 &  3 &  822 &  \cellcolor{yellow}5 &  922 &  \cellcolor{yellow}5 \\
 \hline
 526 &  \cellcolor{yellow}5 &  626 &  3 &  726 &  3 &  826 &  \cellcolor{yellow}5 &  926 &  3 \\
 \hline
 530 &  \cellcolor{yellow}5 &  630 &  \cellcolor{yellow}5 &  730 &  \cellcolor{yellow}5 &  830 &  \cellcolor{yellow}5 &  930 &  \cellcolor{yellow}5 \\
 \hline
 534 &  \cellcolor{yellow}5 &  634 &  \cellcolor{yellow}5 &  734 &  3 &  834 &  3 &  934 &  \cellcolor{yellow}5 \\
 \hline
 538 &  \cellcolor{yellow}5 &  638 &  \cellcolor{yellow}5 &  738 &  \cellcolor{yellow}5 &  838 &  \cellcolor{yellow}5 &  938 &  3 \\
 \hline
 542 &  3 &  642 &  \cellcolor{yellow}5 &  742 &  \cellcolor{red}7 &  842 &  3 &  942 &  3 \\
 \hline
 546 &  3 &  646 &  \cellcolor{yellow}5 &  746 &  3 &  846 &  \cellcolor{yellow}5 &  946 &  \cellcolor{yellow}5 \\
 \hline
 550 &  \cellcolor{yellow}5 &  650 &  3 &  750 &  \cellcolor{yellow}5 &  850 &  \cellcolor{yellow}5 &  950 &  3 \\
 \hline
 554 &  3 &  654 &  3 &  754 &  \cellcolor{yellow}5 &  854 &  3 &  954 &  \cellcolor{yellow}5 \\
 \hline
 558 &  3 &  658 &  \cellcolor{red}7 &  758 &  3 &  858 &  \cellcolor{yellow}5 &  958 &  \cellcolor{yellow}5 \\
 \hline
 562 &  \cellcolor{yellow}5 &  662 &  3 &  762 &  3 &  862 &  \cellcolor{red}7 &  962 &  3 \\
 \hline
 566 &  3 &  666 &  3 &  766 &  \cellcolor{yellow}5 &  866 &  3 &  966 &  \cellcolor{yellow}5 \\
 \hline
 570 &  \cellcolor{yellow}5 &  670 &  \cellcolor{yellow}5 &  770 &  \cellcolor{yellow}5 &  870 &  \cellcolor{yellow}5 &  970 &  \cellcolor{yellow}5 \\
 \hline
 574 &  \cellcolor{yellow}5 &  674 &  3 &  774 &  3 &  874 &  \cellcolor{yellow}5 &  974 &  3 \\
 \hline
 578 &  3 &  678 &  \cellcolor{yellow}5 &  778 &  \cellcolor{yellow}5 &  878 &  3 &  978 &  3 \\
 \hline
 582 &  3 &  682 &  \cellcolor{yellow}5 &  782 &  \cellcolor{yellow}5 &  882 &  3 &  982 &  \cellcolor{yellow}5 \\
 \hline
 586 &  \cellcolor{yellow}5 &  686 &  3 &  786 &  \cellcolor{yellow}5 &  886 &  \cellcolor{yellow}5 &  986 &  \cellcolor{yellow}5 \\
 \hline
 590 &  \cellcolor{yellow}5 &  690 &  \cellcolor{yellow}5 &  790 &  \cellcolor{yellow}5 &  890 &  \cellcolor{yellow}5 &  990 &  \cellcolor{yellow}5 \\
 \hline
 594 &  \cellcolor{yellow}5 &  694 &  \cellcolor{yellow}5 &  794 &  3 &  894 &  \cellcolor{yellow}5 &  994 &  \cellcolor{yellow}5 \\
 \hline
 598 &  \cellcolor{yellow}5 &  698 &  3 &  798 &  3 &  898 &  \cellcolor{yellow}5 &  998 &  3 \\
 \hline
\end{tabular}
\end{table}

\pagebreak

\begin{table}[h]
\centering
\begin{tabular}{||c|c||c|c||c|c||c|c||c|c||}
\hline
$n$ & $C_3(n)$ & $n$ & $C_3(n)$ & $n$ & $C_3(n)$ & $n$ & $C_3(n)$ & $n$ & $C_3(n)$  \\
 \hline
 1002 &  \cellcolor{yellow}5 &  1102 &  \cellcolor{yellow}5 &  1202 &  3 &  1302 &  3 &  1402 &  \cellcolor{yellow}5 \\
 \hline
 1006 &  \cellcolor{yellow}5 &  1106 &  3 &  1206 &  3 &  1306 &  \cellcolor{yellow}5 &  1406 &  3 \\
 \hline
 1010 &  \cellcolor{yellow}5 &  1110 &  \cellcolor{yellow}5 &  1210 &  \cellcolor{yellow}5 &  1310 &  \cellcolor{yellow}5 &  1410 &  \cellcolor{yellow}5 \\
 \hline
 1014 &  3 &  1114 &  \cellcolor{yellow}5 &  1214 &  3 &  1314 &  3 &  1414 &  \cellcolor{yellow}5 \\
 \hline
 1018 &  \cellcolor{yellow}5 &  1118 &  3 &  1218 &  \cellcolor{yellow}5 &  1318 &  \cellcolor{yellow}5 &  1418 &  3 \\
 \hline
 1022 &  3 &  1122 &  \cellcolor{yellow}5 &  1222 &  \cellcolor{red}7 &  1322 &  3 &  1422 &  3 \\
 \hline
 1026 &  3 &  1126 &  \cellcolor{yellow}5 &  1226 &  3 &  1326 &  \cellcolor{yellow}5 &  1426 &  \cellcolor{yellow}5 \\
 \hline
 1030 &  \cellcolor{yellow}5 &  1130 &  \cellcolor{yellow}5 &  1230 &  \cellcolor{yellow}5 &  1330 &  \cellcolor{yellow}5 &  1430 &  \cellcolor{yellow}5 \\
 \hline
 1034 &  \cellcolor{yellow}5 &  1134 &  3 &  1234 &  \cellcolor{yellow}5 &  1334 &  \cellcolor{yellow}5 &  1434 &  \cellcolor{yellow}5 \\
 \hline
 1038 &  \cellcolor{yellow}5 &  1138 &  \cellcolor{yellow}5 &  1238 &  3 &  1338 &  3 &  1438 &  \cellcolor{yellow}5 \\
 \hline
 1042 &  \cellcolor{yellow}5 &  1142 &  3 &  1242 &  \cellcolor{yellow}5 &  1342 &  \cellcolor{yellow}5 &  1442 &  3 \\
 \hline
 1046 &  3 &  1146 &  \cellcolor{yellow}5 &  1246 &  \cellcolor{yellow}5 &  1346 &  3 &  1446 &  3 \\
 \hline
 1050 &  3 &  1150 &  \cellcolor{yellow}5 &  1250 &  3 &  1350 &  3 &  1450 &  \cellcolor{yellow}5 \\
 \hline
 1054 &  \cellcolor{yellow}5 &  1154 &  3 &  1254 &  \cellcolor{yellow}5 &  1354 &  \cellcolor{yellow}5 &  1454 &  3 \\
 \hline
 1058 &  3 &  1158 &  3 &  1258 &  \cellcolor{yellow}5 &  1358 &  3 &  1458 &  3 \\
 \hline
 1062 &  \cellcolor{yellow}5 &  1162 &  \cellcolor{yellow}5 &  1262 &  3 &  1362 &  \cellcolor{yellow}5 &  1462 &  \cellcolor{yellow}5 \\
 \hline
 1066 &  \cellcolor{yellow}5 &  1166 &  \cellcolor{yellow}5 &  1266 &  3 &  1366 &  \cellcolor{yellow}5 &  1466 &  3 \\
 \hline
 1070 &  \cellcolor{yellow}5 &  1170 &  \cellcolor{yellow}5 &  1270 &  \cellcolor{yellow}5 &  1370 &  \cellcolor{yellow}5 &  1470 &  \cellcolor{yellow}5 \\
 \hline
 1074 &  \cellcolor{yellow}5 &  1174 &  \cellcolor{yellow}5 &  1274 &  3 &  1374 &  3 &  1474 &  \cellcolor{yellow}5 \\
 \hline
 1078 &  \cellcolor{yellow}5 &  1178 &  3 &  1278 &  \cellcolor{yellow}5 &  1378 &  \cellcolor{yellow}5 &  1478 &  3 \\
 \hline
 1082 &  3 &  1182 &  \cellcolor{yellow}5 &  1282 &  3 &  1382 &  3 &  1482 &  3 \\
 \hline
 1086 &  3 &  1186 &  \cellcolor{yellow}5 &  1286 &  3 &  1386 &  \cellcolor{yellow}5 &  1486 &  \cellcolor{yellow}5 \\
 \hline
 1090 &  \cellcolor{yellow}5 &  1190 &  \cellcolor{yellow}5 &  1290 &  \cellcolor{yellow}5 &  1390 &  \cellcolor{yellow}5 &  1490 &  \cellcolor{yellow}5 \\
 \hline
 1094 &  3 &  1194 &  3 &  1294 &  \cellcolor{yellow}5 &  1394 &  \cellcolor{yellow}5 &  1494 &  \cellcolor{yellow}5 \\
 \hline
 1098 &  3 &  1198 &  \cellcolor{yellow}5 &  1298 &  3 &  1398 &  \cellcolor{yellow}5 &  1498 &  \cellcolor{yellow}5 \\
 \hline

\end{tabular}
\end{table}

\pagebreak

\begin{table}[h]
\centering
\begin{tabular}{||c|c||c|c||c|c||c|c||c|c||}
\hline
$n$ & $C_3(n)$ & $n$ & $C_3(n)$ & $n$ & $C_3(n)$ & $n$ & $C_3(n)$ & $n$ & $C_3(n)$  \\
 \hline
 1502 &  3 &  1602 &  \cellcolor{yellow}5 &  1702 &  \cellcolor{yellow}5 &  1802 &  \cellcolor{yellow}5 &  1902 &  \cellcolor{yellow}5 \\
 \hline
 1506 &  \cellcolor{yellow}5 &  1606 &  \cellcolor{yellow}5 &  1706 &  3 &  1806 &  3 &  1906 &  \cellcolor{yellow}5 \\
 \hline
 1510 &  \cellcolor{yellow}5 &  1610 &  \cellcolor{yellow}5 &  1710 &  \cellcolor{yellow}5 &  1810 &  \cellcolor{yellow}5 &  1910 &  \cellcolor{yellow}5 \\
 \hline
 1514 &  3 &  1614 &  \cellcolor{yellow}5 &  1714 &  \cellcolor{yellow}5 &  1814 &  3 &  1914 &  \cellcolor{yellow}5 \\
 \hline
 1518 &  \cellcolor{yellow}5 &  1618 &  \cellcolor{yellow}5 &  1718 &  3 &  1818 &  \cellcolor{yellow}5 &  1918 &  \cellcolor{yellow}5 \\
 \hline
 1522 &  \cellcolor{yellow}5 &  1622 &  3 &  1722 &  \cellcolor{yellow}5 &  1822 &  \cellcolor{yellow}5 &  1922 &  3 \\
 \hline
 1526 &  3 &  1626 &  3 &  1726 &  \cellcolor{yellow}5 &  1826 &  \cellcolor{yellow}5 &  1926 &  \cellcolor{yellow}5 \\
 \hline
 1530 &  \cellcolor{yellow}5 &  1630 &  \cellcolor{yellow}5 &  1730 &  \cellcolor{yellow}5 &  1830 &  \cellcolor{yellow}5 &  1930 &  \cellcolor{yellow}5 \\
 \hline
 1534 &  \cellcolor{yellow}5 &  1634 &  3 &  1734 &  3 &  1834 &  \cellcolor{yellow}5 &  1934 &  3 \\
 \hline
 1538 &  3 &  1638 &  3 &  1738 &  \cellcolor{yellow}5 &  1838 &  3 &  1938 &  \cellcolor{yellow}5 \\
 \hline
 1542 &  \cellcolor{yellow}5 &  1642 &  \cellcolor{yellow}5 &  1742 &  3 &  1842 &  3 &  1942 &  \cellcolor{yellow}5 \\
 \hline
 1546 &  \cellcolor{yellow}5 &  1646 &  3 &  1746 &  3 &  1846 &  \cellcolor{yellow}5 &  1946 &  3 \\
 \hline
 1550 &  3 &  1650 &  \cellcolor{yellow}5 &  1750 &  \cellcolor{yellow}5 &  1850 &  3 &  1950 &  3 \\
 \hline
 1554 &  3 &  1654 &  \cellcolor{yellow}5 &  1754 &  3 &  1854 &  3 &  1954 &  \cellcolor{yellow}5 \\
 \hline
 1558 &  \cellcolor{yellow}5 &  1658 &  3 &  1758 &  \cellcolor{yellow}5 &  1858 &  \cellcolor{yellow}5 &  1958 &  \cellcolor{yellow}5 \\
 \hline
 1562 &  \cellcolor{yellow}5 &  1662 &  3 &  1762 &  \cellcolor{yellow}5 &  1862 &  3 &  1962 &  3 \\
 \hline
 1566 &  \cellcolor{yellow}5 &  1666 &  \cellcolor{yellow}5 &  1766 &  3 &  1866 &  \cellcolor{yellow}5 &  1966 &  \cellcolor{yellow}5 \\
 \hline
 1570 &  \cellcolor{yellow}5 &  1670 &  \cellcolor{yellow}5 &  1770 &  \cellcolor{yellow}5 &  1870 &  \cellcolor{yellow}5 &  1970 &  \cellcolor{yellow}5 \\
 \hline
 1574 &  3 &  1674 &  3 &  1774 &  \cellcolor{yellow}5 &  1874 &  3 &  1974 &  \cellcolor{yellow}5 \\
 \hline
 1578 &  \cellcolor{yellow}5 &  1678 &  \cellcolor{yellow}5 &  1778 &  3 &  1878 &  3 &  1978 &  \cellcolor{red}7 \\
 \hline
 1582 &  \cellcolor{yellow}5 &  1682 &  3 &  1782 &  \cellcolor{yellow}5 &  1882 &  \cellcolor{yellow}5 &  1982 &  3 \\
 \hline
 1586 &  3 &  1686 &  \cellcolor{yellow}5 &  1786 &  \cellcolor{yellow}5 &  1886 &  \cellcolor{yellow}5 &  1986 &  3 \\
 \hline
 1590 &  \cellcolor{yellow}5 &  1690 &  \cellcolor{yellow}5 &  1790 &  \cellcolor{yellow}5 &  1890 &  \cellcolor{yellow}5 &  1990 &  \cellcolor{yellow}5 \\
 \hline
 1594 &  \cellcolor{yellow}5 &  1694 &  3 &  1794 &  \cellcolor{yellow}5 &  1894 &  \cellcolor{yellow}5 &  1994 &  3 \\
 \hline
 1598 &  \cellcolor{yellow}5 &  1698 &  3 &  1798 &  \cellcolor{yellow}5 &  1898 &  3 &  1998 &  3 \\
 \hline

\end{tabular}
\end{table}

\end{document}